\documentclass[11pt,a4paper,epsf]{article}
\usepackage{latexsym}
\usepackage[utf8]{inputenc}
\usepackage{amsmath}
\usepackage{amsthm}
\usepackage{amsfonts}
\usepackage{amssymb}
\usepackage{epsfig,xcolor}
\usepackage{nicefrac}
\usepackage[all]{xy}
\usepackage{graphicx}
\usepackage{enumerate}
\newtheorem{theorem}{Theorem}[section]

\newtheorem{proposition}[theorem]{Proposition}
\newtheorem{lemma}[theorem]{Lemma}
\newtheorem{claim}[theorem]{Claim}

\newtheorem{thm}{Theorem}

\theoremstyle{definition}

\newtheorem{remark}[theorem]{Remark}

\newtheorem{definition}[theorem]{Definition}

\def\C{\mathbb C}

\def\Z{\mathbb Z}
\def\T{\mathbb T}
\def\Q{\mathbb Q}

\newcommand{\R}{\mathbb{R}}
\newcommand{\N}{\mathbb{N}}

\newcommand{\cc}{\mathbb{\mathcal C}}

\newcommand{\ca}{\mathcal A}

\newcommand{\cu}{\mathcal U}

\newcommand{\cm}{\mathcal M}

\newcommand{\cl}{\mathcal{L}}

\newcommand{\ct}{\mathcal{T}}

\newcommand{\id}{\textnormal{Id}\,}

\newcommand{\nbd}{neighbourhood }

\newcommand{\area}{\textnormal{Area}\,}

\newcommand{\priv}{\backslash}

\newcommand{\eps}{\varepsilon}
\renewcommand{\phi}{\varphi}

\newcommand{\st}{\textnormal{st}}

\newcommand{\supp}{\text{Supp}\,}

\newcommand{\wdt}[1]{\widetilde{#1}}
\newcommand{\cqfd}{\hfill $\square$ \vspace{0.1cm}\\ }
\newcommand{\sbull}{{\tiny $\bullet$ }}
\newcommand{\ds}{\displaystyle}
\newcommand{\im}{\textnormal{Im}\,}
\newcommand{\re}{\textnormal{Re}\,}

\newcommand{\nf}[2]{{\nicefrac{#1}{#2}}}

\newcommand{\Om}{\Omega}

\newcommand{\cz}{\textnormal{CZ}}

\newcommand{\ind}{\textnormal{index}\,}

\renewcommand{\top}{\textnormal{top}}
\newcommand{\low}{\textnormal{low}}
\newcommand{\lk}{\textnormal{lk}\,}

\def\eps{\varepsilon}

\newcommand{\CC}{{\mathbb C}}

\newcommand{\RR}{{\mathbb R}}
\newcommand{\ZZ}{{\mathbb Z}}
\renewcommand{\P}{{\mathbb P}}

\begin{document}

\title{Squeezing Lagrangian tori in dimension $4$}

\author{R. Hind \thanks{R.H. is partially supported by Simons Foundation grant no. 317510.}  \and E. Opshtein}

\date{\today}

\maketitle

\begin{abstract} Let $\C^2$ be the standard symplectic vector space and $L(a,b) \subset \C^2$ be the product Lagrangian torus, that is, a product of two circles of area $a$ and $b$ in $\C$. We give a complete answer to the question of knowing the minimal ball into which these Lagrangians may be squeezed. The result is that there is full rigidity when $b\leq 2a$, which disappears almost completely when $b>2a$.
\end{abstract}

\begin{section}{Introduction.}

In this paper we investigate the extent to which product Lagrangian tori can be `squeezed' by Hamiltonian diffeomorphisms. To be precise, we determine when such a torus can be mapped into a ball or a polydisk.
To fix notation, we work in the vector space $\RR^4\approx \C^2$ equipped with its standard symplectic form $\omega= \sum_{i=1}^2 dx_i \wedge dy_i$. The Lagrangian product tori are defined by
$$L(a,b) =\{\pi |z_1|^2=a, \pi |z_2|^2=b\} .$$
The open ball of capacity $R$ is given by $$B(R)=\{\pi(|z_1|^2+|z_2|^2)<R\},$$ and our polydisks are defined by $$P(a,b) = \{\pi |z_1|^2 < a, \pi |z_2|^2 < b\} .$$ Hence $L(a,b)$ is the singular part of the boundary of $P(a,b)$.  Up to renormalizing the symplectic form, it is  enough to study the special case of squeezing Lagrangian tori $L(1,x)$, $x\geq 1$. The current paper contains the first results about non-monotone tori, but the monotone case is already known: $L(1,1)$ (or any torus with monotonicity constant $2$) cannot be squeezed into a ball of size $2$, see \cite{ciemon}. Our main result is the following:

\begin{thm}\label{thm:main} For $x \ge 1$, there exists a Hamiltonian diffeomorphism of $\C^2$ that takes $L(1,x)$ into $B(R)$ if and only if $R > \min(1+x,3)$.
\end{thm}
In other terms, the Lagrangian torus $L(1,x)$, that belongs to the boundary of the ball $B(1+x)$, cannot be squeezed into a smaller ball when $x\leq 2$, while it can be squeezed into a ball of size $3+\eps$ if and only if $\eps>0$ when $x\geq 2$. Our result when $x\geq 2$ holds in fact in a slightly more general setting that we review now. For a Lagrangian $L\subset \C^2$, there are two homomorphisms $\Omega, \mu : H_1(L,\ZZ)  \to \RR$ describing how $L$ is embedded in $\C^2$. The first homomorphism is the area class and is defined by $\Omega(e) = [\lambda](e)$ where $\lambda$ is a Liouville form, that is, a primitive of $\omega$. Equivalently, $\Omega(e) = \int _D u^* \omega$ where $D$ is a disk and $u:(D,\partial D) \to (\C^2,L)$ verifies $u_*[\partial D]=e$. The second homomorphism is the Maslov class. If $u:S^1 \to L$ with $u_*[S^1]=e$ then $\mu(e)$ is the Maslov class of the loop of Lagrangian subspaces $T_{u(t)}L \subset \C^2$.

\begin{thm}\label{thm:mainobs} Suppose $L \subset B(R)$ is a Lagrangian torus and $e_1, e_2$ is an integral basis of $H_1(L,\ZZ)$ satisfying
\begin{enumerate}
\item $\Omega(e_1)=1$, $\Omega(e_2) \ge 2$,
\item $\mu(e_1) = \mu(e_2) = 2$.
\end{enumerate}
Then $R >3$.
\end{thm}
Up to our knowledge, although any non-monotone Lagrangian torus in $\C^2$ may be conjectured to be Hamiltonian isotopic to  a product one, there is no proof available at this time. In \cite{gir}, Dimitroglou Rizell-Goodman-Ivrii prove a weaker unknottedness result, namely that any two Lagrangian tori are Lagrangian isotopic. The conjecture would  imply that Theorem \ref{thm:mainobs} is a consequence of Theorem \ref{thm:main}.

We also get a similar result for  Lagrangian tori inside polydiscs.

\begin{thm}\label{thm:polyemb} For $x \ge 1$ and $a \le b$ there exists a Hamiltonian diffeomorphism of $\C^2$  that takes  $L(1,x)$ into
$P(a,b)$ if and only if either $a>2$ or both $a>1$ and $b>x$.
\end{thm}

In particular we see that embeddings of Lagrangian tori into balls or polydisks do not necessarily extend to the corresponding polydisks, where there is a volume obstruction. Even more, there are {\it symplectic} obstructions to embedding a polydisk into a ball which do not obstruct squeezing its Lagrangian "singular boundary" (see \cite{hutchings}, \cite{nelsonc} and compare to Theorem \ref{thm:main}). Notice that in \cite{hl}, the obstruction for embedding a bidisk $P(1,2)$ into a ball was precisely investigated {\it via} studying the singular boundary $L(1,2)$. The present paper shows in particular that this approach is not relevant for studying squeezing of bidisks $P(1,x)$ for $x>2$. Nevertheless our approach can be applied to solve the stabilized polydisk embedding problem $P(1,x) \times \C^n \hookrightarrow B(R) \times \C^n$, see \cite{hindstab}.

\paragraph{Outline of the paper.} Theorems \ref{thm:main} and \ref{thm:polyemb} have two facets. On the one hand, they assert some obstructions for squeezing a Lagrangian torus into a small ball or polydisc. On the other hand, they claim the existence of some embeddings, that we need to construct explicitly. For the obstruction part, our approach goes by studying the behaviour of some curves introduced in \cite{hindker} under a convenient neck-stretching process. We describe the general strategy and set notation in section \ref{sec:gefr}. We then gather in section \ref{sec:lemma} several lemma on properties of the curves in the holomorphic buildings appearing after our neck-stretching, in particular about their indices. In section \ref{sec:proof}, we prove Theorem \ref{thm:mainobs} and the obstruction part of Theorem \ref{thm:main}. Embeddings into balls are actually technically more difficult to study than into polydiscs. In section \ref{sec:polydiscs},  we outline the adjustments needed to deal with embeddings into polydisks.
Finally, we deal with the constructive part of Theorems \ref{thm:main} and \ref{thm:polyemb}  at the same time, by proving the following:
\begin{thm}\label{thm:constr}
\label{embtori} If $x>1$, for any $\lambda>1$,  there exists a Hamiltonian diffeomorphism of $\C^2$ that takes the product torus $L(1,x)$ into $B(3\lambda) \cap P(2\lambda, 2\lambda)$.
\end{thm}
This is described in section \ref{sec:construction}. Our approach is to explicitly write down the image of the Lagrangian torus, which can readily be seen to have Maslov and area classes corrsponding to those of $L(1,x)$. We will then rely on \cite{gir} to show that what we have described is in fact the image of a product torus, rather than a possible exotic nonmonotone torus in $\C^2$ (which conjecturally do not exist).

\end{section}

\begin{section}{Geometric framework}\label{sec:gefr}
The proof of Theorem \ref{thm:main} is based on neck-stretching arguments in the following setting. Let $L\subset B^4(R)$ be an embedded Lagrangian torus in the ball of capacity $R$. We can compactify $B(R)$ to a projective plane $\C\P^2(R)$ with lines of area $R$ and will denote by $S_{\infty}$ the line at infinity. The cotangent bundle of $\T^2$ can be symplectically identified with  $\R^4 \slash \Z^2$ where $\Z^2$ acts by translations in the $(x_1,x_2)$-plane. We fix a very large integer $d$ and an irrational number $S>3d-1$. Then, for $\eps$ small enough, the \nbd
$$
V_{\eps,S}:=\{|y_1|\leq \nf \eps 2,|y_2|\leq \nf{\eps S}2\}
$$
of the zero section symplectically embeds into $B(R)$ as a Weinstein \nbd of $L$, denoted by $V$. Now $V_{\eps,S}$ contains a symplectic bidisk $P(\eps, \eps S)$, which by inclusion contains an ellipsoid $\eps E(1,S)$. By definition $\eps E(1,S) = E(\eps, \eps S)$ and the ellipsoids are defined by
$$E(a,b)=\{\pi (\frac{|z_1|^2}{a}+\frac{|z_2|^2}{b})<1\}.$$
Putting everything together, we therefore have inclusions
$$
\eps E(1,S)\subset V \subset B(R)\subset \C\P^2(R).
$$
We will consider some holomorphic curves in $\C\P^2(R)$ that we wish to stretch along the boundary of $V$. But since this boundary has corners, we first replace $V_{\eps,S}$ by a smooth approximation. Following \cite{hutchingsl}, we define
$$
U_{\eps,S}^p:=\{\|\big(y_1,\frac {y_2}S\big)\|_p<2^\frac 1p\frac \eps 2\}\subset T^*\T^2,
$$
where $\|(\xi_1,\xi_2)\|_p=\big(|\xi_1|^p+|\xi_2|^p\big)^\frac 1p$. Then, $U_{\eps,S}^p$ is a smoothly bounded fiberwise convex subset of $T^*L$ that contains $V_{\eps,S}$, and is close to $V_{\eps,S}$ in the Hausdorff topology when $p$ is large.  For large enough  $p$, we therefore have:
$$
\eps E(1,S)\subset U\subset B(R)\subset \C\P^2(R),
$$
where $U$ is a symplectic embedding of $U_{\eps,S}^p$.  We then have a symplectic cobordism  $X = \CC \P^2(R) \setminus \epsilon E(1,S)$ which supports tame almost-complex structures with  cylindrical ends compatible with the Liouville contact structure on $\partial(\epsilon E(1,S))$. We study finite energy $J$-holomorphic curves in $X$. These will be $J$-holomorphic maps $u: \CC \P^1 \setminus \Gamma \to X$ where $\Gamma$ is a finite set of punctures and $u$ is asymptotic to closed Reeb orbits on $\partial(\epsilon E(1,S))$ at each puncture. We define the degree of these maps to be simply their intersection number with the line at infinity $S_{\infty} = \C\P^2(R)\priv B(R)$, and their asymptotics are iterates of the two closed Reeb orbits on $\partial(\epsilon E(1,S))$, namely (the images of)
$$\gamma_1 = \partial(\epsilon E(1,S)) \cap \{x_2=y_2=0\}$$ and $$\gamma_2 = \partial(\epsilon E(1,S)) \cap \{x_1=y_1=0\}.$$

\paragraph{Some special finite-energy curves in $X$.} Our starting point is an existence theorem for some curves of degree $d$ in this cobordism:
\begin{theorem}[Hind-Kerman \cite{hindker,hk2}] \label{thm:exists} There exists an infinite subset $A\subset \N$ such that, for any $d\in A$, $S>3d-1$ irrational and $\eps E(1,S) \subset \C\P^2(R)$ and for any generic $J$, there exists a rigid finite energy plane $u: \CC \to X$ of degree $d$ asymptotic to the Reeb orbit $\gamma_1^{3d-1}$.
\end{theorem}
This theorem was claimed for any $d$ in \cite[Theorem 2.36]{hindker}. Unfortunately there was a mistake in the proof, which is corrected in \cite{hk2} at the expense of establishing the result only for $d$ belonging to a sequence of natural numbers that diverge to $+\infty$.
In \cite{mcduffpre}, McDuff proves the statement claimed in \cite{hindker}, that the previous result holds with $A=\N$. The version considered here is however enough for the purpose of the current paper.

\paragraph{Neck-stretching.} Let now $J_n$ be a sequence of almost complex structures on $X$ that are cylindrical near $\partial \eps E(1,S)$ and that stretch the neck along $\partial U$. Let $u_n:\C\to X$ be $J_n$-holomorphic finite energy planes provided by Theorem \ref{thm:exists}.  They have degree $d$, and are asymptotic at their puncture to $\gamma_1^{3d-1}$. By \cite{BEHWZ}, this set of curves enjoys a compactness property. In a now well-known sense, our sequence of curves converges (modulo extraction) as $n \to \infty$ to a holomorphic building $B$ made of finite energy holomorphic curves in $S\partial \eps E(1,S)$ , $U\priv \eps E(1,S)$, $S\partial U$ and $\C \P^2\priv U$. (Here $S\partial \eps E(1,S)$ and $S\partial U$ denote the symplectization of $\partial \eps E(1,S)$ and $\partial U$ respectively, with cylindrical almost-complex structures.) These curves have positive and negative ends that are asymptotic to Reeb orbits of $\partial U$ or $\partial \eps E(1,S)$.
 All these ends but one match together pairwise.
The unmatched end is asymptotic to $\gamma_1^{3d-1}$ on $\partial \eps E(1,S)$. Moreover, gluing the different components along their matching ends provide a topological surface which is a bunch of spheres (that may appear because of bubbling phenomenon) and one plane that contains the unmatched end.
We need to gather information on the limit building.
To make our analysis manageable we will identify various sets of limit curves with matching ends and consider them as single components of the limit. Once the ends are identified, we can still talk about the index of such a glued component (see section \ref{sec:indices}).
The identifications are made as follows.
\begin{enumerate}[(I)]
\item The limiting building has a unique curve $u_0$ in its lowest level with negative end asymptotic to $\gamma_1^{3d-1}$. All curves which can be connected to $u_0$ through a chain of curves with matching ends lying in $U \setminus \epsilon E(1,S)$ or the symplectization layers are identified along their matching ends to form our first component $F_0$.
\item Suppose $F_0$ has $T$ unmatched ends. Then the complement of $F_0$ in our limiting building, after identifying matching ends, has exactly $T$ components which we denote $F_1, \dots ,F_T$.
\end{enumerate}

\end{section}

\begin{section}{Preliminaries}\label{sec:lemma}
\begin{subsection}{Reeb orbits and index formulas.}\label{sec:indices}
The Reeb flow on $\partial U$  is conjugated to that of $\partial U_{\eps,S}^p$, which can be easily computed. It preserves the tori $\{(y_1,y_2)=c\}$, and is a linear flow on each such torus, whose slope depends on $c$. When this slope is rational, the torus is foliated by a $1$-dimensional family of periodic orbits.

\begin{proposition} Fix an integral basis of $H_1(L, \Z)$. For each pair of integers $(k,l)$ except $(0,0)$ there is a $1$-parameter family of closed Reeb orbits on $\partial U$ which project to a curve in $L$ in the class $(k,l)$. We say these Reeb orbits are of type $(k,l)$ and denote them by $\gamma_{k,l}$. The orbits are embedded if and only if $k,l$ have no common factor. Orbits $\gamma_{rk,rl}$ are $r$-times covers of the orbits $\gamma_{k,l}$.
\end{proposition}

We now recall the index formulas for the different curves that may appear after the neck stretching process. Before specializing to our situation, let us consider the general setting of holomorphic curves in symplectic cobordisms, and their index \cite{bourgeois,schwarz}.  Let $\dot \Sigma$ be a punctured surface of genus $g$ with $s$ punctures, and $W$ a symplectic cobordism. Recall that this means that $W$ is a symplectic manifold with boundaries $\partial W^+\sqcup \partial W^-$, which are equipped with locally defined outward (inward respectively) pointing Liouville vector fields. The Liouville vector fields define contact forms on the boundary components of $W$, and we assume that their Reeb vector fields are Morse-Bott : the closed Reeb orbits on $\partial W$ may come in smooth families, along which the transverse Poincar\'{e} return maps are non-degenerate. Let also $J$ be an almost complex structure adapted to our cobordism (compatible with the symplectic structure and cylindrical near the ends). Given a finite energy $J$-holomorphic curve $u:\dot \Sigma\to W$, we denote by $\gamma_i^+,\;i=1\dots s_+$ and $\gamma_j^-, \;j=1\dots s_-$ the positive and negative limiting Reeb orbits in $\partial W^+$ and $\partial W^-$, respectively. By the Morse-Bott condition, these asymptotics belong to families of closed Reeb orbits, denoted $S_i^+,S_j^-$. We also fix a symplectic trivialization $\tau$ of $u^*TW$ along these asymptotics. This is possible because the Liouville form induces an orientation of $\partial W$. Then, $c_1^\tau(u^*TW)$ denotes the algebraic number of zeros of a generic section of the vector bundle $\Lambda^2 u^*TW$ which is constant with respect to $\tau$ on the boundary. The formula for the expected dimension of the moduli space of holomorphic curves (moduli reparameterizations) in the same homology class and having the same asymptotics as $u$, called below the index of $u$, is given by
\begin{eqnarray}
\ind(u)=(n-3)\chi(\dot \Sigma)+2c_1^\tau(u^*TX)+\sum_{i=1}^{s_+} \mu_\cz(\gamma_i^+)+\frac 12\dim S_i^++ \notag \\ +
\sum_{j=1}^{s_-} -\mu_\cz(\gamma_j^-)+\frac 12 \dim S_j^-. \label{eq:index}
\end{eqnarray}
In this formula, $\mu_\cz(\gamma)$ represents the classical Conley-Zehnder index of  $\gamma$ when it is non-degenerate (and in this case $\dim S(\gamma)$ vanishes), or the generalized Maslov (Robin-Salamon index) of $\gamma$ in the general case \cite{rosa, gutt}. Note that this dimension formula takes care of the Teichmuller space of $\dot \Sigma$, or its automorphisms group.

As we explained above, we will need for practical computations to (abstractly) glue several curves of the building, and consider the resulting subbuilding as a single entity. We then consider the punctures of this subbuilding to be those of its constituent curves {\it that do not serve as matching ends}. We can then define a notion of index for such a building:
\begin{definition}
Let $B$ be a building made of curves $(u_1,\dots,u_k)$ (in various layers) that match along asymptotic orbits $(\gamma_1,\dots,\gamma_l)$ belonging to spaces $S_1,\dots,S_l$ of closed Reeb orbits (we assume the Morse-Bott situation, where these spaces are manifolds). We define
$$
\ind(B):=\sum_{i=1}^k \ind(u_i)-\sum_{i=1}^l \dim S_i.
$$
\end{definition}
Working with this definition, the following proposition sums up those properties of the index that will be important for us.

\begin{proposition}\label{prop:indglue} The index formula of the buildings has the following properties:
\begin{enumerate}
\item \textnormal{Recursivity:} Let $B$ be a building obtained by gluing different buildings $B_i$ along matching orbits $\gamma_i$ that belong to spaces $S_i$ of Reeb orbits. Then,
$$
\ind (B)=\sum \ind(B_i)-\sum \dim S_i.
$$
\item \textnormal{Computability:} Let $B$ be a holomorphic building and let the underlying curve (after gluing) be $\Sigma$, the positive punctures be $\gamma_i^+$ and negative punctures be $\gamma_i^-$ (recall that those punctures of the constituent curves that have to matched to form $B$ are not considered as punctures of $B$). Then,
$$
\begin{array}{r}
\ds \ind(B)=(n-3)\chi(\dot \Sigma)+2c_1^\tau(B)+\sum_{i=1}^{s_+} \mu_\cz(\gamma_i^+)+\frac 12\dim S_i^++  \\ +
\ds\sum_{j=1}^{s_-} -\mu_\cz(\gamma_j^-)+\frac 12 \dim S_j^-.
\end{array}
$$
Here $c_1^\tau(B)$ simply means the  sum of the $c_1^\tau(u^*TX)$ over the constituent curves of $B$.
\item \textnormal{Continuity:} Let $X$ be a symplectic cobordism, $(J_n)$ a neck-stretching in $X$ (along some hypersurface), and $(B_n)$ a sequence of $J_n$-holomorphic buildings that converge in the sense of \cite{BEHWZ} to a building $B$. Then
$$
\ind(B)=\lim \ind(B_n).
$$
\end{enumerate}
\end{proposition}

Let us now specialize formula (\ref{eq:index}) to our context:
\begin{proposition}\label{prop:indices}
Let $\dot \Sigma$ be a punctured sphere, and $u:\dot \Sigma\to W$ be a $J$-holomorphic map asymptotic to $\gamma_i^+,\gamma_j^-$, $i=1\dots s_+$, $j=1\dots s_-$. Let $s:=s_++s_-$.  We denote the Reeb orbits on $\partial U$ with respect to an integral basis of $L$ consisting of classes with Maslov index $2$.
\begin{enumerate}
\item[a)] If $W=U$, $\ind(u)=2s-2$.
\item[b)] If $W=S\partial U$, $\ind (u)=2s_++s_--2\geq \max(s_-,s_+)$.
\item[c)] If $W=\C\P^2\priv U$ (thus $s=s_-$) and $\gamma_j^-$ is of type $(-k_j,-l_j)$, $\ind(u)=s-2+6d+2\sum(k_j+l_j)$.
\item[d)] If $W=U\priv \eps E$, the $s_-$ negative asymptotics can be further split into $s_1^-$ covers of $\gamma_1$ and $s_2^-$ covers of $\gamma_2$ (we denote $r_i^-,t_j^-$ the multiplicities of these covers). Then,
$$
\ind(u)=2s_+-2-2\sum_{i=1}^{s_1^-} \left( r_i^-+\lfloor \frac{r_i^-}S\rfloor \right) -2\sum_{j=1}^{s_2^-} \left( t_j^-+\lfloor t_j^-S\rfloor\right).
$$
\item[e)] If $X=S\partial E$, $\ind(u)\geq 0$.
\end{enumerate}
\end{proposition}
\noindent {\it Proof:} In dimension $4$, for a punctured sphere, the formula \eqref{eq:index} gives
$$
\ind(u)=s-2+2c_1^\tau(u^*TX)+\sum_{i=1}^{s_+} \mu_\cz(\gamma_i^+)+\frac 12\dim S_i^+ -\sum_{j=1}^{s_-} \mu_\cz(\gamma_j^-)+\frac 12 \dim S_j^-.
$$
In $U$, there is a global Lagrangian distribution $\cl$ given by the vertical distribution of the cotangent bundle. This Lagrangian distribution can be extended to a symplectic trivialization $\tau$ of $u^*TU$, and for this choice, $c_1^\tau(u)=0$. In $\partial U$, each closed orbit comes in a $1$-parameter family and its generalized Maslov index is $\frac 12$. Finally, $U$ is a symplectic cobordism with one positive end $\partial U$, so $s=s_+$. Thus, for $u:\dot \Sigma\to U$,
$$
\ind(u)=s-2+0+\sum_{i=1}^{s_+} \left( \frac 12+\frac 12 \right)=2s-2.
$$

The same choice of $\tau$ can be made in $S\partial U$, and the previous remarks still hold in this setting, except that there are now positive and negative ends.
Thus,  for $u:\dot\Sigma\to S\partial U$, we have
$$
\ind(u)=s-2+0+s_++\sum_{j=1}^{s_-} \left( -\frac 12+\frac 12 \right) =2s_++s_--2.
$$
The maximum principle in $S\partial U$ shows that $s_+\geq 1$, so $\ind(u)\geq s_-$. On the other hand, since no curve $\gamma_{k,l}$ is contractible in $U$ (hence nor in $S\partial U$), $u$ has at least $2$ ends, so $s_++s_-\geq 2$, so $\ind(u)\geq s_+$.

In $W=U\priv \eps E$, we still have the same Lagrangian distribution that we extend to a symplectic trivialization $\tau$ of $u^*TW$. Notice that since $E$ is contractible, we can deform this symplectic trivialization above $E$ (and hence its boundary $\partial E$) so that it coincides with the trivialization coming from $E\subset \R^{2n}$, with the standard trivialization. Relative to this choice, the Conley-Zehnder indices of the closed orbits of $\partial E$ are well-known, and we get for $u:\dot \Sigma\to U\priv \eps E$:
$$
\ind(u)=s-2+s_+-s_--2 \sum_{i=1}^{s_1^-} \left( r_i+\lfloor \frac {r_i}S\rfloor \right) -2 \sum_{j=1}^{s_2^-} \left( t_j+\lfloor t_jS\rfloor \right).
$$

Similarly in $S\partial E$,
$$
\begin{array}{llr}
\ind(u)& = & s-2+s_+-s_- +2 \sum_{i=1}^{s_1^+} \left( r_i^++\lfloor \frac {r_i^+}S\rfloor \right) + 2\sum_{j=1}^{s_2^+} \left( t_j^++\lfloor t_j^+S\rfloor \right)  \\
 & &\left. -2 \sum_{i=1}^{s_1^-} \left( r_i^-+\lfloor \frac {r_i^-}S\rfloor \right) -2 \sum_{j=1}^{s_2^-} \left( t_j^-+\lfloor t_j^-S\rfloor \right) \right. \\

& = & -2 +2 \sum_{i=1}^{s_1^+} \left( r_i^++\lceil \frac {r_i^+}S\rceil \right) + 2\sum_{j=1}^{s_2^+} \left( t_j^++\lceil t_j^+S\rceil \right)  \\
 & &\left. -2 \sum_{i=1}^{s_1^-} \left( r_i^-+\lfloor \frac {r_i^-}S\rfloor \right) -2 \sum_{j=1}^{s_2^-} \left( t_j^-+\lfloor t_j^-S\rfloor \right) \right. \\

 \end{array}
$$
By positivity of the area, we also have $\sum r_i^++\sum t_j^+S\geq \sum r_i^-+\sum t_j^-S$, so $\sum r_i^+ +\sum \lceil t_j^+S \rceil \geq \sum r_i^-+\sum \lfloor t_j^-S \rfloor$ and $\sum \lceil \frac{r_i^+}{S} \rceil +\sum t_j^+ \geq \sum \lfloor \frac{r_i^-}{S} \rfloor +\sum t_j^-$ and one of the inequalities must be strict. Hence we have
$\ind(u)\geq 0$ for any curve $u:\dot \Sigma\to S\partial E$.

Finally, in $\C\P^2\priv U$, we consider the symplectic trivialization over $\partial U$ that comes from the inclusion of $\partial U$ in the affine chart $\C\P^2\priv S_\infty$. Then  $c_1^\tau(u)=6d-6$, $u$ has only negative ends, and if the $i$-th of it  is asymptotic to a Reeb orbit of type $(-k_i,-l_i)$ we find
$$
\ind(u)=s-2+6d+2\left(k_i+l_i\right). \hspace*{4cm}\square
$$
\end{subsection}

\begin{subsection}{Nonnegative index and multiple covers.}\label{mcs}
We recall from section \ref{sec:gefr} that the $F_i$ are subbuildings of the building $B$ obtained by our neck-stretching process, and that their indices are given by the sum of their constituent curves in the different layers, minus the sum of the dimensions of the space of closed Reeb orbits along which these different curves match.  Moreover, $F_0$ plays a special role : it is the connected subbuilding in $U\priv \eps E(1,S)\sqcup S\partial E\sqcup S\partial U$  that is asymptotic to $\gamma_1^{3d-1}$ at its unmatched negative end. It has $T$ positive ends, at which the $F_i$, $i=1\dots T$ are connected.

\begin{lemma}\label{posout} $\mathrm{index}(F_i) \ge 1$ for $1 \le i \le T$.
\end{lemma}

\begin{proof} These components only have one unmatched negatove end in $\partial U$ (corresponding to an end of $F_0$) and hence by propositions \ref{prop:indglue} and \ref{prop:indices}.c), they have odd index. Therefore it suffices to show that $\mathrm{index}(F_i) \ge 0$.
Decompose $F_i$ into subbuildings $\{F_{ij}^+\}_{j=1\dots l_+}\cup \{F_{ij}^-\}_{j=1\dots l_-}$, where the $F_{ij}^+$ are just the constituent curves of $F_i$ in $\C\P^2\priv U$ and the $F_{ij}^-$ their complementary connected subbuildings in $F_i$ (thus $F_{ij}^-$ lies in $U\priv \eps E\sqcup S\partial U\sqcup S\partial E$). Let $s^+_{ij}$ be the number of (negative) ends of $F_{ij}^+$ and $s^-_{ij}$ be the number of (positive) ends of $F_{ij}^+$. Since moreover $F_i$ has only one (unmatched) negative end in $\partial U$, we see that $\sum s^+_{ij} = \sum s^-_{ij} +1$.
We infer by propositions \ref{prop:indglue} and \ref{prop:indices}.a) that the index of $F_{ij}^-$ is $2s_{ij}^--2$.
Now notice that there is no finite energy plane in $U\; (\approx T^*\T^2)$, so $s_{ij}^-\geq 2$ and $\ind (F_{ij}^-)=2s_{ij}^--2\geq s_{ij}^-$. As a result,
$$
\ind(F_i)=\sum\ind (F_{ij}^+)+\sum \ind(F_{ij}^-)-\sum s_{ij}^-\geq \sum \ind( F_{ij}^+),
$$
 and it suffices to prove that $\ind(F_{ij}^+)\geq 0$ $\forall i,j$ to conclude our proof. Let therefore $u$ be a constituent curve of $F_i$ in $\C\P^2\priv U$, and suppose it has degree $d$ and $s$ ends asymptotic to Reeb orbits of type $(-k_i,-l_i)$ as in Proposition \ref{prop:indices}. If $u$ is somewhere injective then for generic $J$ we may assume it has nonnegative index. Otherwise it is a multiple cover of an underlying curve $\tilde{u}$. We may assume that $\tilde{u}$ is somewhere injective and so $\mathrm{index}(\tilde{u}) \ge 0$ for generic almost-complex structures. Suppose this cover is of degree $r$, so $\tilde{u}$ has degree $\tilde{d}=d/r$, and further that $\tilde{u}$ has $\tilde{s}$ negative ends asymptotic to orbits of type $(-\tilde{k}_j,-\tilde{l}_j)$. The Riemann-Hurwitz formula shows that the domain of $\tilde u$ is a punctured sphere, so Proposition \ref{prop:indices} gives
$$
\mathrm{index}(\tilde{u}) = \tilde{s} -2 + 6\tilde{d} + 2\sum_{j=1}^{\tilde{s}} (\tilde{k}_i + \tilde{l}_i).
$$
Hence we have
\begin{equation}\label{eq:rh+index}
\mathrm{index}(u) = r\mathrm{index}(\tilde{u}) + 2(r-1) -(r \tilde{s} - s)\geq 2(r-1) -(r \tilde{s} - s).
\end{equation}
Let now $\phi:S^2\priv \Gamma\to S^2\priv \tilde \Gamma$ be the holomorphic (ramified) covering such that $u=\tilde u\circ \phi$. Removing singularities, $\phi$ extends to a holomorphic map $\Phi:S^2 \to S^2$ that sends $\Gamma$ to $\tilde \Gamma$. Then, $r \tilde s-  s$ represents the total ramification of $\phi$ over the points of $\Gamma$, so Riemann-Hurwitz formula gives
$$
r\tilde s-s=\sum_{c\in \Gamma}  (m_c-1) \leq \sum_{c\in S^2} (m_c-1)=2(r-1).
$$
By \eqref{eq:rh+index}, we see that $\mathrm{index}(u) \ge 0$ as required.
\end{proof}

By proposition \ref{prop:indices}, since $F_0$ has $T$ positive ends and a single negative end asymptotic to $\gamma_1^{3d-1}$ we have $\mathrm{index}(F_0) = 2T - 6d$.

\begin{lemma}\label{posin} $\mathrm{index}(F_0) \ge 0$, that is $T \ge 3d$.
\end{lemma}

\begin{proof} We argue by contradiction and assume that $T<3d$, that is $F_0$ has less than $3d$ positive ends.
Let us consider the curves of $F_0$ in $S \partial E$ which fit together to form a connected component $G_0$ of $F_0$ including the lowest level curve with negative end asymptotic to $\gamma_1^{3d-1}$.
Then let $u_i:S^2\priv \Gamma_i\to U\priv \eps E(1,S)$ be the curve of $F_0$ with a negative end matching the $i$th positive end of $G_0$.

Note that by the maximum principle, each positive end of the $u_i$ is connected through components in $S\partial U$ to positive unmatched ends of $F_0$. And since our building is obtained by degenerating curves of genus $0$, these different positive ends of the $u_i$ are connected to distinct positive unmatched ends of $F_0$. Denoting by $Q_i$ the number of positive ends of $u_i$, we therefore get that $\sum Q_i\leq T\leq 3d-1$. This, in turn, guarantees that no negative end of the $u_i$ is asymptotic to a cover of $\gamma_2$. Indeed, a curve in $U\priv \eps E(1,S)$
is (possibly a multiple cover of) a somewhere injective curve $\tilde u$ is $U\priv \eps E$ with $s_+$ ends with index
$$
\ind(\tilde u)=2s_+-2-2\sum_{i=1}^{s_1^+} \left( r_i+\lfloor \frac {r_i}S\rfloor \right)  -2 \sum_{j=1}^{s_2^+} \left( t_j+\lfloor t_j S\rfloor \right)\geq 0.
$$  Since $s_+\leq 3d-1$ and $S>3d-1$, we see that the $t_j$ must vanish, so $\tilde u$ and therefore $u$ itself has no negative end asymptotic to a cover of $\gamma_2$.
Altogether, the curves $u_i$ therefore verify the following. They have $Q_i$ positive ends with $\sum Q_i= Q\leq 3d-1$, and they have a negative end asymptotic to a $\gamma_1^{q_i}$ which is matched with a positive end of $G_0$.
For area reasons, we then see that $\sum q_i\geq 3d-1$. In total, we therefore have
$$
\sum Q_i\leq 3d-1\leq \sum q_i,
$$
so there exists an $i$ such that $Q_i\leq q_i$. We henceforth denote this curve by $u$, and let $q$ and $Q$ the corresponding numbers; hence $q\geq Q$. The index of this curve is
$$
\ind(u)=2Q-2-2 \sum  \left( r_i+\lfloor \frac {r_i}S\rfloor\right)\leq 2Q-2-2q<0.
$$
Thus $u$ must be an $r$-covering of a somewhere injective curve $\tilde u:S^2\priv \tilde \Gamma\to U\priv \eps E(1,S)$ with say $\tilde Q\leq Q<3d$ positive ends, $\tilde s_-$ negative ends, the $i$-th of which is asymptotic to $\gamma_1^{r_i}$ (none of them is asymptotic to a cover of $\gamma_2$), and non-negative index given by

$$
\mathrm{index}(\tilde{u}) = 2\tilde{Q} -2 - \sum_{i=1}^{\tilde{s}^-} (2\tilde{r}_i + 2 \lfloor \frac{\tilde{r}_i}{S} \rfloor).
$$
Hence $\tilde{Q} \ge 1+ \sum \tilde{r}_i$. Suppose that the end of $u$ asymptotic to $\gamma_1^q$ covers $t$ times an end of $\tilde{u}$ asymptotic to $\gamma_1^{\tilde{q}}$, so $t\tilde{q}=q$. Consider also the ramified covering $\phi:S^2\priv \Gamma\to S^2\priv \Gamma$ defined by $u=\tilde u\circ \phi$, and remove its singularities to get a map $\phi:S^2\to S^2$. We recall that $\Gamma$ splits as $\Gamma^+\cup \Gamma^-$, where $\Gamma^+$ are the positive ends of $u$ and $\Gamma^-$ the negative ones. One of the negative ends, say $c_0$, has order $t$. The Riemann-Hurwitz formula, together with the facts that  $\tilde{Q} \ge 1+ \sum \tilde{r}_i$, $t\leq r$, and $q\geq Q$, 
then give:
$$
\begin{array}{ll}
2(r-1)=\sum_{c\in S^2}(m_c-1)&\geq \sum_{c\in \Gamma^+}(m_c-1)\; + (t-1)\\ & =r\tilde Q-Q+t-1\\
& \geq r(1+\sum \tilde r_i)-Q+t-1\\
 & \geq r(1+\tilde q)-Q+t-1\\
 & =r+r\tilde q-Q+t-1\\
 & = r+t\tilde q-Q+(r-t)\tilde q+t-1\\
 & \geq r+q-Q+r-t+t-1\\
 & \geq 2r-1,
\end{array}
$$
which is a contradiction.
%
\end{proof}

\end{subsection}

\begin{subsection}{Holomorphic planes in $\C^2\priv U$}
We recall the inclusions $L\subset U\subset B^4(R)$, with $U$ symplectomorphic to $U_{\eps,S}^p$. We consider here the situation where $L$ is the image by a Hamiltonian diffeomorphism of $\C^2$ of the product torus $L(1,x)$ with $x\geq 1$. We moreover assume in this paragraph that $x\in \Q$. The aim of this section is the following result:
\begin{lemma}\label{le:inC2}
Under the above hypothesis, for a generic almost complex structure and $s \ge 1$, there is no genus $0$ finite energy curve in $\C^2\priv U$ with $s$ negative ends, deformation index at least $s$ and area strictly less than $1$.
\end{lemma}

\begin{remark}\label{areadefn} To avoid complicating our formulas with terms of order $\epsilon$, we will define the {\it area} of a finite energy curve in $\C^2\priv U$ to be the area of the closed surface formed by topologically gluing a half-cylinder in $U$ to each end of $u$. The open end of the cylinder is asymptotic to a limiting Reeb orbit matching the corresponding end of $u$ and the boundary of the cylinder is a curve on $L$. Note that this cylinder is {\it symplectic}, so starting with a curve in $\C^2\priv U$ of positive genuine area, the area we define here still remains positive, and this is all we will care about in the sequel.
\end{remark}

Before proving Lemma \ref{le:inC2} we derive the property for the standard Lagrangian torus $L(1,x)\subset \C^2$ and some cylindrical almost complex structure.

\begin{lemma}\label{le:inC2st} Let $U$ be a \nbd of $L(1,x)\subset \C^2$ (with $x\geq 1$), symplectomorphic as before to $U_{\eps,S}^p$. Let $J$ be an almost complex structure on $\C^2\priv U$ cylindrical near $\partial U$ and such that $J$ coincides with the standard complex structure $i$ near the line $\{z_2=0\}$. Then, there is no genus $0$ finite energy curve in $\C^2\priv U$ with $s$ negative ends, deformation index at least $s$ and area strictly less than $1$.
\end{lemma}
\noindent{\it Proof:} Let $u:\C\P^1\priv\{z_1,\dots,z_s\}\to \C^2\priv U$ be a genus $0$ $J$-holomorphic curve with finite energy and index at least $s$.
Its ends are asymptotic to orbits of type $(-k_j, -l_j)$ as in section \ref{sec:lemma} where we now use the standard basis of $H_1(L(1,x), \Z)$.
Then,
$$
\ind(u)=s-2+2\sum_{j=1}^s (k_j+l_j)\geq s
$$
 so $\sum k_j+l_j\geq 1$. On the other hand,
$$
\begin{array}{ll}
\area (u)=\sum_{j=1}^s k_j +l_jx &= \sum_{j=1}^sk_j+l_l+l_j(x-1) \\
 & \geq 1+(x-1)\sum_{j=1}^s  l_j.
\end{array}
$$
 Now the sum of the $l_j$'s represents the intersection number between $u$ and $\{z_2=0\}$ (parameterized in the obvious way by the $z_1$-coordinate). Since $J=J_\st$ near this line, and $u$ is holomorphic, each intersection point between these two curves count positively, so $\sum l_j\geq 0$. Since $x\geq 1$, we indeed get $\area(u)\geq 1$. \cqfd

\noindent{\it Proof of lemma \ref{le:inC2}:}
Arguing by contradiction, since there exists a Hamiltonian diffeomorphism $f$ mapping $L(1,x)$ to $L$, if such an almost-complex structure and finite energy curve exists then we can pull-back using $f$ to find holomorphic curves of area less than $1$ asymptotic to a neighborhood $U$ of the product torus. Hence it suffices to work with  $L=L(1,x)$. By Lemma \ref{le:inC2st} we can find at least one almost-complex structure $J_1$ for which no such curves exist. Moreover, since the only constraints on $J_1$ appear near $\{z_2=0\}$ we may assume any genericity properties of $J_1$ with respect to curves asymptotic to $\partial U$.

We need two facts about moduli spaces of finite energy curves in $\C^2\priv U$. We denote by ${\cal J}$ the collection of compatible almost-complex structures on $\C^2\priv U$.

\begin{theorem}[Ivrii, \cite{ivrii}, section 2.4, Wendl, \cite{wendl}.]\label{immersed} Given $J \in {\cal J}$, immersed $J$-holomorphic finite energy curves with index at least the number of negative ends are regular. That is, the normal Cauchy-Riemann operator is surjective and our curves appear in a family of the expected dimension. Such curves are also regular in their moduli space of curves with fixed asymptotic limits (rather than allowing the limits to move in the family of Reeb orbits).
\end{theorem}

\begin{theorem}[see Zehmisch, \cite{zeh}, Oh and Zhu, \cite{ohzhu}]\label{regular} There exists a subset of ${\cal J}_I \subset {\cal J}$ of the second category such that if $J \in {\cal J}_I$, in any moduli space of somewhere injective $J$-holomorphic curves the collection of singular (that is, non-immersed) curves form a stratified subset of codimension $2$.
\end{theorem}

The papers \cite{zeh} and \cite{ohzhu} are concerned with compact holomorphic curves, but \cite{zeh} generalizes directly. A proof that moduli spaces of dimension $0$ or $1$ generically contain only immersed curves also appears in \cite{wendl}, and this case will be our main focus.

Let ${\cal J}_R \subset {\cal J}$ be the collection of almost-complex structures on $\C^2\priv U$ which are regular for all of the (countably many) moduli spaces of finite energy curves. We will work with ${\cal J}_0 = {\cal J}_R \cap {\cal J}_I \subset {\cal J}$, which is again a set of the second category.

Given this, we define $\ca$ to be the infimum of the areas of $J$-holomorphic curves having index at least the number of negative ends and $J \in {\cal J}_0$. By contradiction we are assuming $\ca<1$. But since $x \in \Q$ there are only finitely many possible areas less than $1$ which can be realized by holomorphic curves and so $\ca>0$ is realized by, say, a $J_0$-holomorphic curve $u$. Amongst all choices for $u$ we choose a curve with the minimal number of negative ends $s$. We may further assume that $u$ is somewhere injective from the following lemma.

\begin{lemma}\label{wlogsi} Let $v$ be a finite energy curve in $\C^2\priv U$ with $t$ negative ends and $\ind{v} \ge t$. Suppose that $v$ is a multiple cover of a curve $\tilde{v}$ with $\tilde{t}$ negative ends. Then $\ind(\tilde{v}) \ge \tilde{t}$.
\end{lemma}

\begin{proof} Suppose the ends of $\tilde{v}$ are asymptotic to orbits of type $(-\tilde{m}_i, -\tilde{k}_i)$ and the cover is of degree $r$. Then we have
$$\ind(\tilde{v}) = \tilde{t} -2 + 2\sum_{i=1}^{\tilde{t}}( \tilde{m}_i + \tilde{k}_i)$$ and
$$\ind(v) = t -2 + 2r\sum_{i=1}^{\tilde{t}}( \tilde{m}_i + \tilde{k}_i)$$
$$=t-2 + r(\ind(\tilde{v}) - \tilde{t} +2).$$
Thus $$ r(\ind(\tilde{v}) - \tilde{t} +2) -2 \ge 0$$ and $$ \ind(\tilde{v}) - \tilde{t} \ge \frac{2}{r}-2 >-2.$$ As $ \ind(\tilde{v}) - \tilde{t}$ is even this gives our inequality as required. \cqfd


Now, by the definition of ${\cal J}_0$ the $J_0$-holomorphic curve $u$ is regular, and by Theorem \ref{immersed} we may further assume it is immersed. Denote its asymptotic limits by $\sigma_1, \dots, \sigma_s$. In the case when $\ind(u)>s$ we also fix $N= \frac{1}{2}(\ind(u) -s)$ points $p_1, \dots, p_N$ in the range of the injective points of $u$ (recall from Proposition \ref{prop:indices}.c) that $N$ must be an integer).

Let $J_t$, $0 \le t \le 1$ be a family of almost-complex structures interpolating between $J_0$ and $J_1$ and define the universal moduli space
$$
\cm:= \left \{   \begin{array}{ll}

(t,u) \quad \text{such that} & u:\C\P^1\priv \{z_1,\dots,z_s\}\to \C^2\priv U,\\ & \overline\partial_{J_t}u=0,\\
 & u \text{ somewhere injective} \\
  &\text{image}(u)\cap p_i \neq \emptyset \quad \text{for all} \quad i,\\
   & u \text{ is asymptotic to } \sigma_i \text{ at }z_i
  \end{array}\right\} \slash \sim
  $$
where we quotient by reparameterizations of the domain. We note that $\cm$ has virtual dimension $1$. Indeed, fixing the asymptotic limits reduces the virtual dimension by $s$, and the fixed points further reduce the dimension by $2N$.

Similarly to the above we assume that $\{J_t \}$ is regular for $\cm$ (so $\cm$ has dimension $1$) and further is regular for all unconstrained moduli spaces of somewhere injective finite energy curves. We also assume $\{J_t \}$ is generic in the sense of Theorem \ref{immersed} for singular curves, so they all appear in our moduli spaces only in codimension $2$ (and hence not at all in $\cm$).

Given all of this, since curves in $\cm$ are immersed they are also regular by Theorem \ref{regular} (which also holds when the index is a constrained index for curves passing through fixed points). Therefore the map $\cm \to [0,1]$ is a submersion. By Lemma \ref{le:inC2st} the fiber over $1$ is empty and so we will arrive at a contradiction if we can show $\cm$ to be compact.

To this end, let $(t_n, u_n) \in \cm$ and suppose that $t_n \to t_{\infty}$. We claim that a subsequence of the $u_n$ converges to a $J_{t_{\infty}}$-holomorphic finite energy plane $u_{\infty}$ such that $(t_\infty, u_\infty) \in \cm$.
By \cite{BEHWZ},  some subsequence of $u_n$ converges in a suitable sense to a $J_{t_\infty}$-holomorphic building. This building consists of top level curves in $\C^2\priv U$ and lower level curves in $S\partial U$. For simplicity we will gather lower level curves with matching asymptotics and consider them as a single curve. We can also assume by adding trivial lower cylinders that there are no unmatched negative ends of the top level curves. Hence, there are $s$ unmatched negative ends to the lower level curves which are asymptotic to $\sigma_1, \dots, \sigma_s$. Let $L_\top$  be the number of top level curves and $L_\low$ the number of low level components (after our identifications). We call $I_i,s_i$, $i=1,\dots ,L_\top$ the  (unconstrained) index and number of negative ends of the $i$-th top level curve, and $(J_i,r_i,t_i)$, $i=1,\dots,L_\low$ the index, number of positive and number of negative ends of the $i$-th low-level component. Then by proposition \ref{prop:indices}.b),
 $$
 J_i=2r_i+t_i-2.
 $$
 Since all ends of the upper level curves match,
 $$
 \sum_{i=1}^{L_\top} s_i=\sum_{i=1}^{L_\low} r_i,
 $$
 and since the total number of unmatched ends is $s$,
 $$
 \sum_{i=1}^{L_\low} t_i=s.
 $$
 We can associate a graph to a holomorphic building by adding a vertex for each curve and asymptotic limit (just one vertex when an asymptotic limit is a matching asymptotic between two curves), and an edge between the vertex for each curve and its asymptotics. As we take limits of curves of genus $0$, this graph is a tree, hence has Euler characteristic $1$, so we have
 $$
 L_\top+L_\low+\sum_{i=1}^{L_\top}s_i+\sum_{i=1}^{L_\low} t_i-\sum_{i=1}^{L_\top}s_i-\sum_{i=1}^{L_\low}(r_i+t_i)=L_\top+L_\low-\sum_{i=1}^{L_\low} r_i=1.
 $$
 Finally, the index of the limiting building, namely the sum of the indices of the constitutent curves minus matching conditions, equals the index of the original curve, which by assumption is at least $s$. Thus, taking all previous equalities into account,
 $$
 \begin{array}{ll}
\ds s\leq  \sum_{i=1}^{L_\top}I_i+\sum_{i=1}^{L_\low}J_i-\sum_{i=1}^{L_\top} s_i& \ds =\sum_{i=1}^{L_\top} (I_i-s_i)+\sum_{i=1}^{L_\low}  (2r_i+t_i-2)\\ &\ds =\sum_{i=1}^{L_\top} (I_i-s_i)+2(L_\top+L_\low-1)+s-2L_\low\\
  &\ds  =s+2(L_\top-1)+\sum_{i=1}^{L_\top} (I_i-s_i) \\
  & \ds =s-2+\sum_{i=1}^{L_\top} (I_i-s_i+2)
\end{array}
 $$
 Thus,
 $$
 \sum_{i=1}^{L_\top} (I_i-s_i+2)\geq 2.
 $$
By proposition \ref{prop:indices}.c) the differences $I_i-s_i$ are even, so at least one of them is  non-negative, and the corresponding top level curve $v$ therefore has index at least the number of its negative ends. As it appears as part of the limit, the area of $v$ is at most the common area $\ca$ of the curves $u_n$. By Lemma \ref{wlogsi} we may assume $v$ is somewhere injective and our assumptions on $\{J_t \}$ imply the curve $v$ appears in a universal moduli space containing some immersed curves. By Theorem \ref{regular} such an immersed curve is regular and so persists if we deform to an almost-complex structure $J \in {\cal J}_0$, and by minimality of $\ca$ among the areas of such curves, the area of $v$ is at least $\ca$ and hence $\ca$ exactly. Thus there is exactly one top level curve (as $v$ occupies all of the area) which must therefore intersect the points $p_i$. By minimality of $s$ we see that $v$ has at least $s$ negative ends, but as the limit is of curves of genus $0$ and there are no holomorphic planes in $U$ (there are no contractible Reeb orbits) the curve $v$ must have exactly $s$ negative ends. It follows that the lower level curves are cylinders which for action reasons must be trivial cylinders asymptotic to the $\sigma_1, \dots, \sigma_s$. We conclude that $v$ has the correct asymptotics and $(t_{\infty}, v) \in \cm$ as required.
\end{proof}

\end{subsection}

\end{section}

\begin{section}{Proof of the obstruction part of theorem \ref{thm:main}}\label{sec:proof}

We now gather together the information from the previous section to prove the obstructions claimed in theorem \ref{thm:main}.

\paragraph{Some restrictions on the $F_i$:}
 By Lemmas  \ref{posout} and \ref{posin}, the component $F_0$ has $\mathrm{index}(F_0) \ge 0$ and the $F_i$ for $1 \le i \le T$ have $\mathrm{index}(F_i) \ge 1$. As there are $T$ remaining ends to match and the sum of the indices minus matching is $0$, we conclude that $\mathrm{index}(F_0) = 0$ (and hence $T=3d$) and $\mathrm{index}(F_i) = 1$ for $i \ge 1$.

By proposition \ref{prop:indices}, the index equality for the $F_i$ with $i \ge 1$ says
\begin{equation}\label{eq:indfi}
1=\mathrm{index}(F_i)=6d_i + 2(m_i + k_i) -1,
\end{equation}
 where $F_i$ has total degree $d_i$ and is asymptotic to orbits of type $(-m_i, -k_i)$.
Meanwhile the action in $B(R)$ of orbits of type $(-m_i,-k_i)$ is $m_i + k_i x$ (recalling Remark \ref{areadefn}).
Therefore by Stokes' Theorem we have $$\area(F_i) = Rd_i +  (m_i + k_i x).$$
In view of \eqref{eq:indfi} we get
$$
\area(F_i) = \big(R-3 \big)d_i + (k_i(x-1) +1).
$$

\paragraph{The case $x\ge 2$ (proof of theorem \ref{thm:mainobs}):}
We argue by contradiction, and assume that $R<3$ and $x \ge 2$. As $\mathrm{area}(F_i) >0$ this means $k_i \ge 0$. But as all limiting orbits $\gamma_{-m_i, -k_i}$ bound the component $F_0$ in $U$ they represent a trivial homology class, and so $\sum m_i = \sum k_i =0$. Thus $k_i =0$ for all $i$ and the $3d$ components $F_i$ are all asymptotic to orbits of type $(-m_i,0)$.

Now we use Stokes' Theorem to calculate the area of $F_0$ in $U$. This has $3d$ positive ends asymptotic to orbits $\gamma_{-m_i,0}$ and a negative end on $\gamma_1^{3d-1}$. Hence from our description of $U$ in section \ref{sec:gefr} we get
$$
\mathrm{area}(F_0) = \sum_i |m_i| \frac{\epsilon}{2} - (3d-1)\epsilon.
$$
As $\sum m_i =0$ we have $\sum_{m_i >0} m_i \ge 3d-1$. Focus for a moment on the components $F_i$ with $m_i>0$. By the index formula \eqref{eq:indfi} their index is $1=6d_i+2m_i-1$, so their degrees $d_i$ vanish, while  $m_i=1$. Since the sum of these degrees is at least $3d-1$, there are exactly $3d-1$ such components (since there are $3d$ components in total, and one of them at least must have $m_i<0$).
The final picture for the building is therefore a component $F_0$ with $3d$ positive ends, $3d-1$ of which are asymptotic to orbits $\gamma_{-1,0}$ that match with components $F_i$ of degree $0$, and one positive end asymptotic to an orbit $\gamma_{3d-1,0}$ that matches with a component (say $F_{3d}$) of degree $d$. The positive area of this last component is $Rd-(3d-1)$,
so $R>3-\nf 1d$.
Taking $d$ large gives us the result.



\paragraph{The case $x<2$:} We again argue by contradiction, and assume now that $R<1+x$. Note that by Weinstein's \nbd theorem, it is enough to prove our result for $x\in \Q$. Note also that if all $k_i$ vanish, the same proof as the above shows that $R\geq 3$, which is already a contradiction. Hence, there must be planes asymptotic to orbits of type $(m,k)$ with $k\neq 0$, and in particular there must be such a plane $F$ asymptotic to an orbit $\gamma_{m,k}$ with $k<0$. Let $d$ be the degree of $F$. Then,
$$
1=\ind(F)=6d+2(m+k)-1
$$
and
$$
\area(F)=Rd+(m+kx).
$$
From the first equation we get $m=1-3d-k$, and substituting in the second one gives
$$
\area(F)=Rd+(1-3d-k)+kx=(R-3)d+k(x-1)+1>0.
$$
Since $k\leq -1$, $x <2$ and $R < 1+x$, we get
$$
(2-x)d<2-x
$$
and so $d=0$. Thus $F$ is a plane with one negative end, has index $1$, lies in $B^4(R)\priv U\subset \C^2\priv U$, and verifies $m=1-k$, so $\area(F)=1+k(x-1)<1$. This is in contradiction with Lemma \ref{le:inC2} (recall that $x\in \Q$).\cqfd
\end{section}

\begin{section}{Proof of the obstruction part of theorem \ref{thm:polyemb}.}\label{sec:polydiscs}

We briefly outline the adjustments required the establish the obstruction part of Theorem \ref{thm:polyemb}. Note that there does not exist an embedding $L(1,x) \hookrightarrow P(a,b)$ when $a<1$ since by \cite{chsc}, Proposition 2.1, the Lagrangian torus $L(1,x)$ has displacement energy $1$. We still argue by contradiction, assuming that $a<2$ and $b<x$.
The proof of Theorem \ref{thm:polyemb} proceeds similarly to that of Theorem \ref{thm:main} except now we compactify $P(a,b)$ to a copy of $S^2 \times S^2$ with factors having areas $a$ and $b$. The analogue of Proposition \ref{prop:indices} is that the deformation index of a finite energy curve $u$ of bidegree $(d_1,d_2)$ asymptotic to Reeb orbits of type $(-m_i,-k_i)$ is given by
$$\mathrm{index}(u) = s -2 + 4(d_1 + d_2) + 2\sum_{i=1}^{s} (m_i + k_i).$$

As before, we consider the situation
$$
\eps E(1,S)\subset U\subset S^2(a)\times S^2(b),
$$
where $U$ is our Weinstein \nbd of $L$. We now use the existence of  $J$-holomorphic planes $u:\C\P^1\priv \{\infty\}\to S^2\times S^2\priv \eps E(1,S)$ of bidegree $(d,1)$, asymptotic to $\gamma_1^{2d+1}$, which exist for $d$ arbitrarily large by \cite{hindker}, and we stretch the neck of $\partial U$. As previously, we split the limit building $B$ into subbuildings $F_0,F_1,\dots,F_T$, where $F_0$ is the maximal connected subbuilding of $B$ in $S\partial \eps E$, $U\priv \eps E$ and $S\partial U$ attached to the negative unmatched end, while $F_1,\dots F_t$ are the connected subbuildings attached to the $T$ positive ends of $F_0$. Arguing as for balls, we get the following:
\begin{enumerate}
\item $\ind(F_i)\geq 1$. The proof is exactly the same as lemma \ref{posout}, replacing the formula for the index of $\tilde u$ by the correct formula in $S^2\times S^2\priv U$.
\item $\ind(F_0)\geq 0$. Since $F_0$ lies in $U$, this is exactly lemma \ref{posin}. Since now $\ind(F_0)=2T-4(d+1)$, we conclude that $T\geq 2(d+1)$.
\end{enumerate}
Then again we get restrictions on the $F_i$: $\ind(F_0)=0$, $\ind(F_i)=1$ and $T=2(d+1)$. Suppose $F_i$ is asymptotic to an orbit of type $(-m_i, -k_i)$ and has bidegree $(d_i^1, d_i^2)$. The index and area formula  for the $F_i$ now give
$$
\begin{array}{l}
1=\ind(F_i)=-1+4(d_i^1+d_i^2)+2(m_i+k_i),\\
\area(F_i)=d_1a+d_2b+m_i+k_ix,
\end{array}
$$
so
$$
\area(F_i)=d_i^1(a-2)+d_i^2(b-2)+k_i(x-1)+1>0,
$$
 while the degrees verify $\sum d_i^1=d$, $\sum d_i^2=1$.

 When $x>2$, taking into account our assumptions $a<2$, $b<x$,  we see that the component with $d_i^2=1$ verifies
 $(b-1)+k_i(x-1)>0$, so $k_i\geq 0$ because $b<x$. Meanwhile the ones with $d_i^2=0$ verify $k_i(x-1)+1> 0$, so $k_i\geq 0$ because $x\geq 2$. Thus $k_i\geq 0$ for all $i$, and since $\sum k_i=0$, we get all $k_i=0$, so all Reeb asymptotics are of the form $\gamma_{(m,0)}$. Exactly the same argument as in section \ref{sec:proof} then shows that $2d+1$ of the $F_i$ have bidegree $(0,0)$ and are asymptotic to $\gamma_{(-1,0)}$ while one has bidegree $(d,1)$ and asymptotic  $\gamma_{(2d+1,0)}$. The area of this subbuilding is
 $$
 ad+b-2d-1=(a-2)d+b-1>0.
 $$
 Since this inequality has to hold for arbitrarily large $d$, it contradicts our assumption $a<2$.

 When $x\leq 2$  we see as above that if all $k_i$ vanish, we get a contradiction with $a<2$. Thus at least one of the $F_i$ must be asymptotic to $\gamma_{(-m,-k)}$, $k<0$. Its area is
 $$
 \ca=d_1(a-2)+d_2(b-2)+k(x-1)+1, \hspace{2cm} k\leq -1.
 $$
Now $1+k(x-1)<1$, so lemma \ref{le:inC2} rules out the possibility that $d_1=d_2=0$. If $d_1\geq 1$, since $d_2\leq 1$ (and since $a<2$ by assumption), we get 
$$
a-2+b-2-x+2=a-2+b-x\geq 0,
$$
which contradicts the fact that $a<2$ and $b<x$. If $d_1=0$, $d_2=1$, we get $b-2-x+2\geq 0$ so $b\geq x$. \cqfd

\end{section}

\begin{section}{Squeezing Lagrangian tori: proof of theorem \ref{thm:constr}}\label{sec:construction}
The construction is slightly delicate, so we proceed in several steps. We first construct a Lagrangian torus close to $P(2,2)$ or to $B(3)$, with several  properties that we use in a second time to show that this Lagrangian torus is in fact Hamiltonian isotopic to a product torus.
\begin{proposition}\label{prop:existence}
Let  $x>2$ and $\cu$ be an arbitrary \nbd of $B(3)\cap P(2,2)$. There exists a Lagrangian torus $L\subset \cu$ which has an integral basis $(e_1,e_2)$ of its $1$-dimensional homology with $\mu(e_1)=\mu(e_2)=2$, $\Om(e_1)=x$, $\Om(e_2)=1$. This Lagrangian torus bounds  a solid torus $\Sigma\subset \C^2$
with symplectic meridian discs having boundary in the class $e_2$ and whose characteristic foliation is by closed leaves.

One can further impose that there exists $a,b\in \C$ such that:
\begin{itemize}
\item[a)] $\Sigma\subset \C^2\priv\{z=a\}\cup\{w=b\}$ (hence the same holds for $L$),
\item[b)] $\lk(e_2,\{z=a\})=\lk(e_2,\{w=b\})=0$,
\item[c)] $\lk(e_1,\{z=a\})=1$, $\lk(e_1,\{w=b\})=-1$.
\end{itemize}
\end{proposition}
\noindent{\it Proof:} Let $A\geq 1$ be a real number to be chosen more specifically later. Let $\gamma$ be the immersed closed loop in the $z=(x_1,y_1)$-plane represented in Figure \ref{fig:gamma}. It is contained in the square $S:=\{0<x_1<A\; ,\; 0<y_1<A\}$ and approximates an $n$-times cover of $\partial S$. Label the uppermost horizontal segment close to $\{y_1=0\}$ by $H$ and the rightmost vertical segment close to $\{x_1=A\}$ by $V$ (see figure \ref{fig:gamma}). For technical reasons, we chose $0<\eps\ll 1$ and assume that $H\subset \{y_1<\eps\}$ and $V\subset \{x_1\geq A-\nf \eps 2\}$, while $\gamma$ lies in the $\eps$-strip around $\partial S$.  Removing $H$ from $\gamma$ provides an embedded curve spiraling inwards and winding $n$ times around $\partial S$. Replacing $H$ results in $(n-1)$ self-intersections, all near $H \cap V$.

Let $T:=\partial D$ be the unit circle in $\C_w$, $w=x_2+iy_2$. The product of $\gamma$ by $T$ is an immersed Lagrangian torus in $\C^2$. Denoting by $\tilde e_1=[\gamma\times\{*\}]\in H_1(\gamma\times T)$ and $e_2=[\{*\}\times T]$, we have $\mu(\tilde e_1)=2n$, $\mu(e_2)=2$, $\Om(\tilde e_1)\approx nA^2$ and $\Om(e_2)=1$.

\paragraph{Step 1:} We first construct a Lagrangian torus $L$ as required, except for the linking conditions b) and c). Let $K(w)$ be a
Hamiltonian function which displaces the circle $T$ in a disc of area $2+\eps$ and satisfies $0\leq G\leq 1+\eps$. Let $\chi:\R\to [0,1]$ be a  smooth function with $\chi(t)=1$ for $t\geq A-\eps$ and $0\leq \chi'(t)\leq 1$ (recall $A>1$). Let also $\Om\subset \C$ be defined by $\Om:=Q_1\cup Q_2\cup Q_3$, with
$$
\begin{array}{l}
Q_1:=\{\eps<x_1<A-\eps\}\\
Q_2:=\{x_1\geq A-\nf \eps 2\}\\
Q_3:=\{A-\eps <x_1<A\}\cap \{0<y_1<\eps\}.
\end{array}
$$
Note that $H\subset Q_1\cup Q_3$, $V\subset Q_2\cup Q_3$ and $Q_1\cap Q_2=\emptyset$ (see figure \ref{fig:gamma}). We leave the reader check the following straightforward facts:
\begin{claim}The function $f:\Om\subset \C\to \R$ defined by $f_{|Q_1}:=\chi(x_1)$, $f_{|Q_2}:=1-\chi(y_1)$ and $f_{|Q_3}:=1$ is smooth, and its  Hamiltonian flow verifies:
\begin{enumerate}
\item $\Phi^t_{f|Q_3}=\id$,
\item $\Phi^t_f(H\cap Q_1)\subset \subset Q_1\cap \{0<y_1<A-\eps\}$ and $\Phi^t_f(H)\cap(\gamma\priv H)=\emptyset$,
\item $\Phi^t_f(V\cap Q_2) \subset Q_2\cap \{A-\nf \eps 2<x_1<A+1\}$ and $\Phi^t_f(V)\cap (\gamma\priv V)=\emptyset$.
\end{enumerate}
\end{claim}
We finally define the function $K(z,w):=f(z)G(w)$ on $\Om\times \C\subset \C^2$, and
$$
L:=\gamma\priv(H\cup V)\times T\,\cup\, \Phi^1_K(H\cup V\times T).
$$
\begin{figure}[h!]
\begin{center}
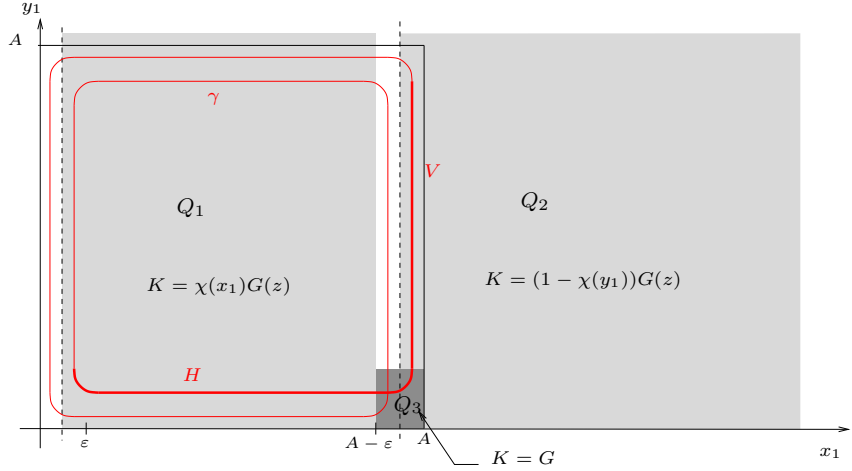
\end{center}
\caption{The curve $\gamma$, and the function $K$.}
\label{fig:gamma}
\end{figure}

Since $f=0$ near $\partial (H\cup V)$ (consider $H\cup V$ as a connected arc), $L$ is an immersed torus whose area and Maslov classes coincide with those of $\gamma\times T$. Notice also that the Hamiltonian vector field of $K$ is $G(w)\vec X_f(z)+f(z)\vec X_G(w)$, and since $0\leq f\leq 1$, and $0\leq G\leq 1+\eps$, for any subset $P\subset\Om\times \C$, $\pi_z(\Phi^t_K(P))\subset \cup_{s\leq (1+\eps)t}\Phi^s_f(\pi_zP)$ and $\pi_w\Phi^t_K(P)\subset \cup_{s\leq (1+\eps)t} \Phi^s_G(\pi_wP)$. Hence, by points (3) and (4) of the above claim, $L$ has no self-intersection in $Q_1\times \C$, nor in $Q_2\times \C$. Moreover, since $f=1$ on $Q_3$, $\Phi^1_K=\id\times \Phi^1_G$ on $Q_3\times \C$, so
$\Phi^1_K(H\cup V\times T)\cap \gamma\times T\cap Q_3=\emptyset$. Therefore, $L$ is embedded.  The homology of $L$ is generated by $\tilde e_1:=[\tilde \gamma]:=[\gamma\priv (H\cup V)\times \{*\}\cup \Phi^1_K(H\cup V\times \{*\})]$ and $e_2:=[\{*\}\times T]$ with $\{*\}$ in $\gamma\priv \supp(f)$.
We have $\mu(e_2)=2$ and $\mu(\tilde e_1)=2n$, so $(e_1:=\tilde e_1-(n-1)e_2,e_2)$ is a Maslov $2$ integral basis of $H_1(L)$, with areas
$$
\ca(e_2)=1,\hspace{0,5cm} \ca(e_1)=n(A^2-\delta)-(n-1),
$$
where $\delta$ is a term that can be chosen arbitrarily small. Given $x$, one may find $A$ arbitrarily close to $1$, $\delta\ll 1$ and $n$ large enough, so that $\ca(e_1)=x$, which we assume henceforth. Thus, $L$ is an embedded Lagrangian torus, with a Maslov two basis of areas $x$ and $1$, which lies in $D(A(A+1))\times D(2)$, which is an arbitrary \nbd of $P(2,2)$  provided $A$ is close to $1$.

\paragraph{Step 2:} We now show that $L$ can in fact be taken into an arbitrary \nbd of $B(3) \cap P(2,2)$. Let $\eps\ll 1$ and $\psi:[0,A+1]\times [0,A]\to D(A(A+1)+\eps)$ be such that $\psi([0,A]\times [0,A])\subset D(A^2+\eps)$ and $\psi([A,A+1]\times [0,\rho])\subset D(A^2+\rho+\eps)$ $\forall \rho$. Such a symplectic diffeomorphism is easy to construct by stretching the part $[A,A+1]\times[0,A]$ horizontally by a factor $4$,  then wrapping the long strip around the square $[0,A]^2$ (see figure \ref{fig:psi1}, or  \cite[Lemma 3.3.3]{schl}.).

\begin{figure}[h!]
\begin{center}
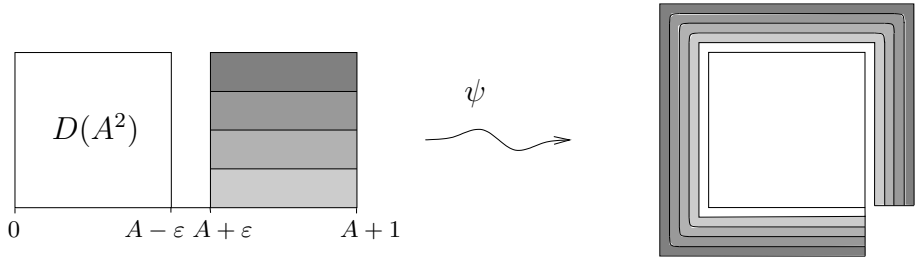
\end{center}
\caption{The projection of $L$ onto the $z$-plane, and its image by $\psi$.}
\label{fig:psi1}
\end{figure}

We claim that $\psi\times \id(L)$ lies close to $B(3)$ provided $A$ is close to $1$. Indeed,
\begin{itemize}
\item[\sbull] if $(z,w)\in L$, with $z\in[0,A]\times [0,A]$, then $\psi\times \id(z,w)=(z',w')$ verifies $|z'|\in D(A^2+\eps)$ and $|w'|\in D(2)$, so $(z',w')\in B(A^2+2+\eps)$.
\item[\sbull] If $(z,w)\in L$ with $z\in [A,A+1]\times\{\rho\}$, then $(z,w)\in \Phi^1_K(V\times T)$. Since
$$
\begin{array}{ll}
\vec X_K(z,w)& =f(z)\vec X_G(w)+G(w)\vec X_f(z)\\
 & =[1-\chi(y_1)]\vec X_G(w)+G(w)\vec X_f(z) \text{ in }Q_2,
\end{array}
$$
$\vec X_K$ preserves the hyperplanes $\{y_1=c\}$ in $Q_2\times \C$, so $(z,w)=\Phi^1_K(z',w')$ with $\im z=\im z'=\rho$. Then,
$$
\pi_w\Phi^t_K(z,w)=\Phi^{(1-\chi(\rho))t}_G(w),
$$
so we may assume $w\in D(2-\chi(\rho))$, while $z\in [A,A+1]\times \{\rho\}$. As a result, $\psi\times \id(z,w)\in D(A^2+\rho+\eps)\times D(2-\chi(\rho))\subset B(A^2+2+\rho-\chi(\rho)+\eps)$. Notice now that since $A$ is close to $1$, $\chi(\rho)$ can be chosen at $\cc^0$-distance $\eps$ from the identity, and this shows that $\psi(z,w)$ lies arbitrarily close to $B(3)$.
\end{itemize}

\paragraph{Step 3:} We now slightly alter the previous construction to achieve the linking condition. Note first that taking $a$ in $\{\re z<\eps\}$ in such a way that $\gamma$ winds around $a$ exactly once achieves the correct linking between $L$ and $\{z=a\}$. The problem is therefore with the line of the type $\{w=b\}$.

First, fix $b\in \C$ in the complement of $\cup \Phi^t_G(T)$, in $D(2+\eps)$. In particular, since $\Phi^1_G$ displaces $T$, $b$ is in the complement of $D\cup \phi^1_G(D)$, where $T=\partial D$. There exists a Hamiltonian diffeomorphism $\Psi$ with support in $D(2+\eps)$, disjoint from $D\cup \Phi^1_G(D)$ and such that $\forall p\in T$, the path $\Phi_G^t (p)$ for $0 \le t \le 1$ and $\Psi \circ \Phi_G^{2-t}(p)$ for $1 \le t \le 2$ winds around $b$ with winding $-1$. Such a diffeomorphism is not easy to describe {\it verbatim}, but easy to draw (see figure \ref{fig:Psi}).

\begin{figure}[h!]
\begin{center}
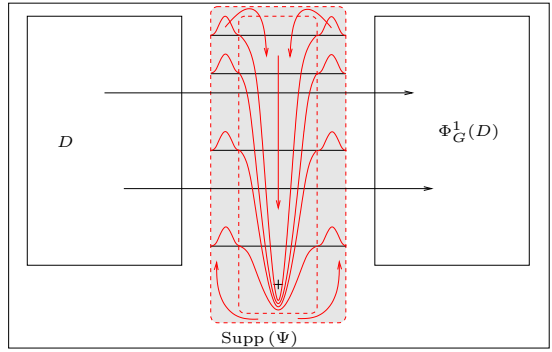
\end{center}
\caption{The diffeomorphism $\Psi$.}
\label{fig:Psi}
{\footnotesize The red oriented curves represent the flow lines of $\Psi$. The red unoriented lines represent the image of the horizontal foliation. The support of $\Psi$ can be made arbitrarily thin, so that the whole picture fits into a disc of area $2+\eps$.}
\end{figure}

We now define
$$
L_\Psi:=\gamma\priv (H\cup V)\times T\cup \Phi^1_K(H\times T)\cup  \Psi(\Phi^1_K(V\times T)),
$$
where $K(z,w)=f(z)G(w)$ as before and by abuse of notation we are writing $\Psi(z,w) = (z, \Psi(w))$.

Let us explain why $L_\Psi$ satisfies all our requirements. Note first that $L_\Psi$ is obtained by $L$ from cutting the tube $\ct:=\Phi^1_K(V\times T)$ and pasting $\ct_{\Psi}:=\Psi(\Phi^1_K(V\times T))$. As $\pi_w(\Phi^1_K( \partial V \times K)) = D\cup \Phi^1_G(D)$ is disjoint from the support of $\Psi$ we have that $\ct_{\Psi}$ differs from $\ct$ by a Hamiltonian generated by a function with compact support. Moreover, since the corresponding Hamiltonian flow is parallel to the $w$-plane, the flow of $\ct$ remains disjoint from the remainder of $L$ and so the generating function can be extended to be identically $0$ near the remainder of $L$. We conclude that
$L_\Psi$ differs from $L$ by a Hamiltonian diffeomorphism and in particular $L_\Psi$ has the same Maslov and area classes as $L$.

We now claim that $L_\Psi$ has the correct linking with the lines $\{z=a\}$ and $\{w=b\}$. It is clear for the former, while the latter needs a justification. First, since $e_2$ is represented by $\{*\}\times T$ for suitable $\{*\}$, and since $b\notin D$, $\lk(e_2,\{w=b\})=0$.  Fix now $w_0\in T$. A representative of the class $\tilde e_1=e_1+(n-1)e_2$ is simply
$$
\gamma\priv (H\cup V)\times \{w_0\}\cup \Phi^1_K(H\times \{w_0\})\cup \Psi(\Phi^1_K(V\times \{w_0\})).
$$
Parametrize $H$ by $\rho\mapsto \rho+ ic$, for $\rho\in [\rho_0,\rho_1]$, $\rho_0<\eps$ and $\rho_1>A$ and $c<\eps$. Then,
\begin{equation}
\Phi^1_K(\rho+ic,w_0)=(\rho+i\tilde c(\rho),\Phi_G^{\chi(\rho)}(w_0)). \label{eq:H}
\end{equation}
Similarly $V$ can be parameterized by $\rho \mapsto b + i\rho$ and
\begin{equation}
\Psi(\Phi^1_K(b +i\rho,w_0) = \big(\tilde b(\rho)+i\rho,\Psi\circ \Phi^{1-\chi(\rho)}_G(w_0)\big) \label{eq:V}
\end{equation}

By choice of $\Psi$, we therefore see that $\lk(\tilde e_1,\{w=b\})=-1$, and since $e_1=\tilde e_1-(n-1)e_2$, that $\lk(e_1,\{w=b\})=-1$. In view of the particular form of $\Psi$, whose support can clearly be given any arbitrarily small area, the arguments that showed that $L$ can be taken into an arbitrary \nbd of $B(3)$ go through. Finally, $L_\Psi$ bounds the solid torus
$$
\Sigma_\Psi:=\gamma\priv (H\cup V)\times D\cup \Phi^1_K(H\times D)\cup \Psi(\Phi^1_K(V\times D)).
$$
The characteristic foliation of $\gamma\times D$ is by curves $\gamma\times \{*\}$, and  $\Sigma_\Psi$ is obtained from $\gamma\times D$  by explicit  Hamiltonian diffeomorphisms on various pieces, and therefore the characteristic leaves remain closed as required.  \cqfd

\end{section}

Let us now prove that the Lagrangian tori that we have just produced are Hamiltonian isotopic to a product torus.

\begin{proposition}\label{prop:isotopy}
A Lagrangian torus that satisfies the requirements of proposition \ref{prop:existence} is Hamiltonian isotopic to the product torus $L(1,x)$ in $\C^2$.
\end{proposition}
\noindent {\it Proof:} Let $L,\Sigma,a,b$ be as in Proposition \ref{prop:existence}, but observe that without loss of generality $a=b=0$. We denote $X:=\C^2\priv\{z=0\}\cup\{w=0\}$. We wish to prove that $L$ is Hamiltonian isotopic to $L(x,1)=S^1(x)\times S^1(1)$ in $\C^2$. We proceed in two steps.



\paragraph{Step 1 (Lagrangian isotopy):} We first reposition the product torus $L(x,1)$ in $X$ and then find a Lagrangian isotopy inside $X$ between $L$ and the repositioned torus.

Consider the linear symplectomorphism $$A:(z,w) \mapsto \frac{1}{\sqrt{1-\lambda^2}}(z+\lambda \overline{w}, w + \lambda \overline{z})$$ where $\lambda$ is a constant satisfying $\sqrt{\frac{1}{x}} < \lambda <1$. We observe that $A(L(x,1)) \subset X$ and denote the image by $L_0$. We denote the image of the standard basis of $H_1(L(x,1))$ by $f_1, f_2 \in H_1(L_0)$, so $\mu(f_1) = \mu(f_2)=2$ and $\Omega(f_1)=x, \Omega(f_2)=1$. Also $\lk(f_1,\{z=0\})=1$, $\lk(f_1,\{w=0\})=-1$ and $\lk(f_2,\{z=0\})=0$, $\lk(f_2,\{w=0\})=0$. Moreover the solid torus $\Sigma_0 = A(\{ \pi|z|^2 = x, \pi |w|^2 \le 1\})$ also lies in $X$ and has trivial characteristic foliation.

The existence of a Lagrangian isotopy between $L$ and $L_0$ given the existence of $\Sigma$ and $\Sigma_0$ has already been established in  \cite[Theorem 6.1]{gir},  (see also \cite[Proposition 3.4.6]{ivrii}). For the sake of clarity and completeness, we briefly recall an outline of the argument here.

Let us fix characteristic leaves (called core circles) $\Gamma \subset \Sigma$ and $\Gamma_0:=A(S^1(x)\times \{0\}) \subset \Sigma_0$, and  a meridian disk $D$ in $\Sigma$. The return map of the characteristic foliation on $D$ in $\Sigma$ being the identity, any isotopy of circles in $D$ starting from $\partial D$ defines a Lagrangian isotopy of $L$, by considering the suspension of these loops by the characteristic foliation in $\Sigma$.
Hence an isotopy of loops that shrinks $\partial D$ to a small loop around $D\cap \Gamma$ gives a Lagrangian isotopy from $L$ to a small neighborhood of $\Gamma$ in $\Sigma$.

Next, since the linking numbers about the axes are the same, $\Gamma$ and $\Gamma_0$ are smoothly isotopic in $X$, and the smooth isotopy extends to a {\it symplectic} isotopy $\psi_t$ (not Hamiltonian), defined in a \nbd of $\Gamma$, that brings $\Gamma$ to $\Gamma_0$.
Since these two curves are characteristic leaves, the image $\Sigma' = \psi_1(\Sigma)$  is an embedding of $S^1\times D(\eps)$  which is tangent to $\Sigma_0$ along $\Gamma_0$.

There may not necessarily be a symplectic isotopy mapping $\Sigma'$ into $\Sigma_0$, indeed the obstruction is the relative winding number of the characteristic foliations in $\Sigma'$ and $\Sigma_0$ with respect to a trivialization of the symplectic normal bundle to $\Gamma_0 = \psi_1(\Gamma)$. However this relative winding can be corrected by a small perturbation of $\Sigma_0$ as follows.



We identify a neighborhood of $\Gamma_0$ in $X$ with a neighborhood of $S^1 \times \{0\} \times \{ 0 \}$ in $S^1 \times \R \times \C$ with the product symplectic form from $S^1 \times \R = T^* S^1$ and $\C$. We can make this identification such that $\Sigma_0$ is $S^1 \times \{ 0 \} \times \C$. Given this choose a smooth function $\chi:\R^+\to \R^+$ with small support, small $\cc^0$-norm, linear with slope $N$ near $0$, and define
$$
\wdt\Sigma_0:=\{(\theta,\chi(|w|^2),w)\in S^1\times \R\times \C,\; |w|<1\}.
$$
The characteristic foliation of $\wdt \Sigma_0$ is given by the curves $\theta\mapsto(\theta,\chi(|w|^2), we^{i\chi'(|w|^2)\theta})$. Thus close to $\Gamma_0$ the characteristic foliation of $\wdt \Sigma_0$ can be arranged to have a winding $N$ equal to that of $\Sigma'$. It follows that we can find a symplectic isotopy $\zeta_t$ taking a neighborhood of $\Gamma_0$ in $\Sigma'$ into $\wdt \Sigma_0$.

The image $\wdt L = \zeta_1(\psi_1(L))$ of $L$ is now a circle of characteristic leaves in $\wdt \Sigma_0$ intersecting a meridinal disk $D_0$ in a small circle $\sigma_0$ about $\Gamma_0$. It remains then to find a Lagrangian isotopy between $\wdt L$ and $L_0$. Now, the return map of a meridinal disk generated by characteristic leaves in $\wdt \Sigma_0$ is not the identity in the region where $\chi$ has nonintegral slope; however the characteristic foliation is tangent to the tori $\{|w|=c\}\cap \wdt\Sigma_0$ in this region. Hence we construct our isotopy between $L_0$ and $\wdt L$ as above by choosing an isotopy between the circles $\partial D_0$ and $\sigma_0$ in $D_0$ which coincides with the circles $\{|w|=c\}$ in the region where the slope of $\chi$ varies.

\paragraph{Step 2 (Hamiltonian isotopy):} We now use classical transformations on $X$ to modify our Lagrangian isotopy $L_t$ to become Hamiltonian.

Let $(e_1^t,e_2^t)$ the continuous determination of basis of $H_1(L_t,\Z)$ that starts at $(e_1,e_2)\in H_1(L,\Z)$, and denote by $\alpha_t,\beta_t$ the corresponding symplectic actions. It is not hard to see that $e_2^0 \in H_1(L_0)$ is the class of a meridian circle $A(\{z=*\})$ and so $\beta_1 = \beta_0 =1$. Since Lagrangian isotopies preserve the Maslov class, $e_1^0 \in H_1(L_0)$ is a Maslov $2$ class with the same linking as $e_1$. There is only one such class, namely $f_1$, and so $\alpha_1 = \alpha_0 =x$. Our goal is to deform the isotopy $L_t$ inside $X$ relative to $L_{0,1}$ such that $\alpha_t = x$ and $\beta_t =1$ for all $t$. The isotopy will then be Hamiltonian as required.



Recall that $X=\C^2\priv \{z=0\}\cup\{w=0\}$ is symplectomorphic to a subset of $T^*\T^2$, for instance to $\{(\theta_1,\theta_2,p_1,p_2)\in \T^2\times (\R^*_+)^2\}$.

Start first by applying dilations $d_\lambda:(\theta_1,\theta_2,p_1,p_2)\mapsto (\theta_1,\theta_2,\lambda p_1,\lambda p_2)$. These dilations are conformally symplectic so they preserve the class of Lagrangian submanfiolds. They correspond to standard dilations $(\lambda z,\lambda w)$ in $\C^2$ and thus $L_t':=d_{\nf 1 {\beta_t}}L_t$ is a Lagrangian isotopy from $L$ to $L_0$ in $X$ with actions  $\alpha_t'$ and $\beta_t'\equiv 1$.

To correct the action $\alpha_t'$ we use translations $\tau^1_c:(\theta_1,\theta_2,p_1,p_2)\mapsto (\theta_1,\theta_2,p_1+c,p_2)$ or $\tau^2_c:(\theta_1,\theta_2,p_1,p_2)\mapsto (\theta_1,\theta_2,p_1,p_2+c)$. These transformations are symplectic, and they preserve the set $\{p_1,p_2>0\}$, and hence $X$, provided $c>0$. In $\C^2$, they correspond to inflating either the line $\{z=0\}$ or $\{w=0\}$, that is removing these lines and pasting $D(c)\times \C$ in their place. Because of the linking condition, we see that the area class of $\tau_c^1L_t'$ is $(\alpha_t'+c,1)$ while that of $\tau_c^2L_t'$ is $(\alpha_t'-c,1)$. Thus, defining $L_t'':=\tau^1_{x-\alpha_t'}L_t'$ when $\alpha_t'\leq x$ and $L_t'':=\tau^2_{\alpha_t'-x}L_t'$ when $\alpha_t'\geq x$, we get a Lagrangian isotopy from $L$ to $L_0$ by Lagrangian submanifolds with fixed action in $\C^2$, and hence a Hamiltonian isotopy as required. \cqfd

\end{document}